\documentclass[12pt]{amsart}
\usepackage{amssymb}
\usepackage{amsmath}
\numberwithin{equation}{section}

\usepackage{graphicx} 
\usepackage{epstopdf}

\newtheorem{thm}{Theorem}[section]
\newtheorem{lemma}[thm]{Lemma}
\newtheorem{cor}[thm]{Corollary}
\newtheorem{prop}[thm]{Proposition}

\theoremstyle{definition}

\theoremstyle{remark}
\newtheorem{remark}[thm]{Remark}

\newcommand{\sgn}{{\rm sgn}}

\renewcommand{\S}{\mathfrak S}
\newcommand{\s}{\sigma}

\newcommand\C{{\mathbb{C}}}

\newcommand\Z{{\mathbb{Z}}}
\newcommand\F{{\mathbb{F}}}
\newcommand\PP{{\mathbb{P}}}

\newcommand\wh{\widehat}

\newcommand\bq{\begin{equation}}
\newcommand\eq{\end{equation}}
\newcommand\beq{\begin{eqnarray*}}
\newcommand\eeq{\end{eqnarray*}}
\newcommand\ben{\begin{enumerate}}
\newcommand\een{\end{enumerate}}
\newcommand\bit{\begin{itemize}}
\newcommand\eit{\end{itemize}}

\newcommand\des{{\rm des}}
\newcommand\exc{{\rm exc}}

\newcommand\maj{{\rm maj}}
\newcommand\comaj{{\rm comaj}}

\newcommand\sg{{\mathfrak S}}

\newcommand\Exd{{\rm DEX}}

\newcommand\fix{{\rm fix}}
\newcommand\ch{{\rm ch}}

\newcommand\x{{\mathbf x}}

\newcommand\Exp{{\rm Exp}}

\newcommand\bnd{{\rm bnd}}

\def\wh{\widehat}
\def\hz{\hat 0}

\def\ov{\overline}
\def\ttn{T_{t,n}}

\def\rh{\tilde{H}}

\def\zz{{\mathbb Z}}
\def\nn{{\mathbb N}}

\def\pp{{\mathbb P}}

\def\xx{{\mathbf x}}

\def\si{\sigma}
\def\fff{{\rm f}}

\begin{document}

\title[Poset homology and $q$-Eulerian polynomials]
{Poset homology of Rees products, and $q$-Eulerian polynomials}
\author[Shareshian]{John Shareshian$^1$}
\address{Department of Mathematics, Washington University, St. Louis, MO 63130}
\thanks{$^{1}$Supported in part by NSF Grants
 DMS 0300483 and DMS 0604233, and the Mittag-Leffler Institute}
\email{shareshi@math.wustl.edu}

\author[Wachs]{Michelle L. Wachs$^2$}
\address{Department of Mathematics, University of Miami, Coral Gables, FL 33124}
\email{wachs@math.miami.edu}
\thanks{$^{2}$Supported in part by NSF Grants
DMS 0302310 and DMS 0604562, and the Mittag-Leffler Institute}


\date{October 30, 2008}

\dedicatory{Dedicated to Anders  Bj\"orner on the occasion of his 60th birthday}

\begin{abstract} The notion of  Rees product of posets was introduced by
Bj\"orner and Welker in \cite{bw}, where they study connections between poset topology and commutative algebra.  Bj\"orner and Welker  conjectured and Jonsson \cite{jo} proved  that the dimension of the top homology of the Rees product of the truncated Boolean algebra $B_n \setminus \{0\}$ and the $n$-chain $C_n$ is equal to the number of derangements in  the symmetric group $\mathfrak S_n$.  Here we prove a refinement of this result, which involves the Eulerian numbers, and a   $q$-analog of both the refinement and the original conjecture, which comes from replacing the Boolean algebra by the lattice of subspaces of the $n$-dimensional vector space over the $q$ element field, and  involves the $(\maj,\exc)$-$q$-Eulerian polynomials studied in  previous papers of the authors \cite{sw,ShWa}.  Equivariant versions of the refinement and the original conjecture are also proved, as are type BC versions (in the sense of Coxeter groups) of the original conjecture and its $q$-analog.\end{abstract}

\maketitle

\vbox{
\tableofcontents
}

\section{Introduction and statement of main results}

In their study of connections between topology of order complexes and commutative algebra in \cite{bw}, Bj\"orner and Welker introduced the notion of Rees product of posets, which is a combinatorial analog of the Rees construction for semigroup algebras.  They stated a  conjecture that the M\"obius invariant of a certain family of Rees product posets is given by the derangement numbers. Our investigation of this conjecture led to a surprising new $q$-analog of the classical formula for the exponential generating function of the Eulerian polynomials, which we proved in \cite{ShWa} by establishing certain quasisymmetric function identities.  In this paper, we return to the original conjecture (which was first proved by Jonsson \cite{jo}).  We prove a refinement  of the conjecture, which involves Eulerian polynomials, and we prove a q-analog and  equivariant version of both the conjecture and its refinement, thereby connecting poset topology to the subjects studied in our earlier paper.

  The terminology used in this paper is explained briefly here and more fully in  Section~\ref{prelimsec}.  All posets are assumed to be finite.

Given ranked posets $P,Q$ with respective rank functions $r_P,r_Q$, the {\it Rees product} $P \ast Q$ is the poset whose underlying set is
\[
\{(p,q) \in P \times Q:r_P(p) \geq r_Q(q) \},
\]
with order relation given by $(p_1,q_1) \leq (p_2,q_2)$ if and only if all of the conditions
\begin{itemize}
\item $p_1 \leq_P p_2$,
\item $q_1 \leq_Q q_2$, and
\item $r_P(p_1)-r_P(p_2) \geq r_Q(q_1)-r_Q(q_2)$
\end{itemize}
hold.  In other words, $(p_2,q_2)$ covers $(p_1,q_1)$ in $P \ast Q$ if and only if $p_2$ covers $p_1$  in $P$ and either $q_2=q_1$ or $q_2$ covers $q_1$ in $Q$. 

\vspace{.2in}\begin{center}\includegraphics[width=3.5in]{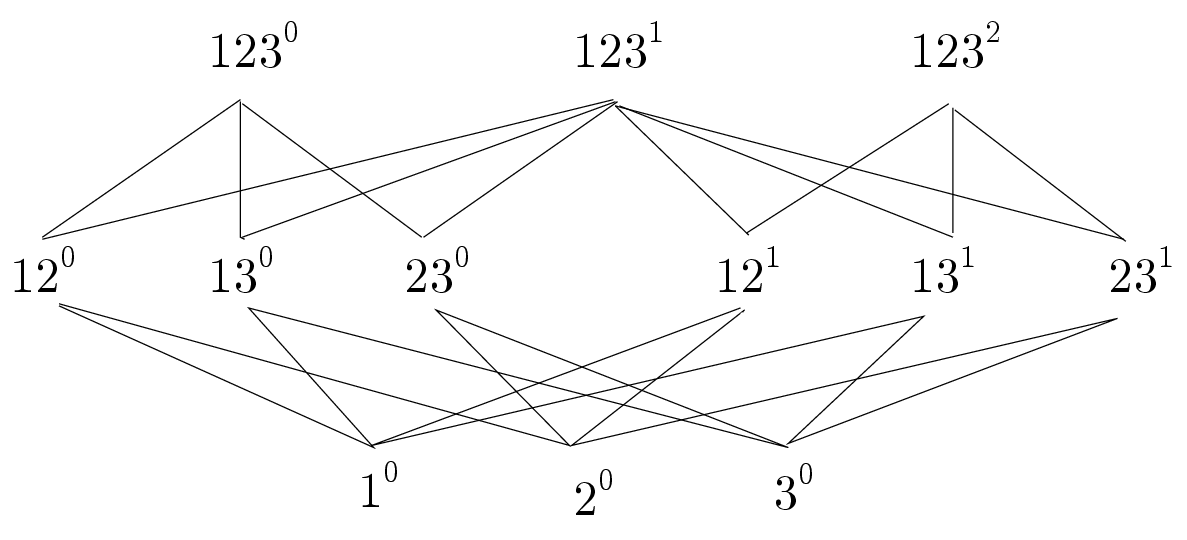}\end{center}
\centerline{{\bf Figure 1.}  $(B_3\setminus\{\emptyset\} ) * C_3$}
\vspace{.2in}

Let $B_n$ be the Boolean algebra on the set $[n]:=\{1,\ldots,n\}$ and   $C_n$ be the chain $\{0<1<\ldots<n-1\}$.  This paper is concerned with the Rees product $(B_n\setminus\{\emptyset\} ) * C_n$ and various analogs.   The Hasse diagram of  $(B_3\setminus\{\emptyset\} ) * C_3$  is given in Figure 1 (the pair $(S,j)$ is written as $S^j$ with set brackets omitted).

Recall that for a poset $P$, the {\it order complex} $\Delta P$ is the abstract simplicial complex whose vertices are the elements of $P$ and whose $k$-simplices are totally ordered subsets of size $k+1$ from $P$.  The (reduced) homology of $P$ is given by $\rh_k(P):= \rh_k(\Delta P;\C)$. Bj\"orner and Welker \cite{bw} prove that if $P$ is the Rees product of Cohen-Macaulay posets then $P$ is Cohen-Macaulay, which means that homology of each interval and principal lower and upper order ideal of $P$ is concentrated in its top dimension.    Hence
$(B_n \setminus \{\emptyset\}) * C_n$ is Cohen-Macaulay, since both $ B_n \setminus \{\emptyset\}$ and $C_n$ are.

For any poset $P$ with a minimal element $\hat 0$, let $P^-$ denote the truncated poset $P \setminus \{\hat 0\}$.    The theorem of Jonsson as conjectured by Bj\"orner and Welker in \cite{bw} is as follows.

\begin{thm}[Jonsson \cite{jo}] \label{jj}
We have
\begin{equation*}
\dim \rh_{n-1}(B_n^- \ast C_n)=d_n,
\end{equation*}
where $d_n$ is the number of derangements (fixed-point-free elements) in the symmetric group $\mathfrak S_n$.
\end{thm}

Our refinement of Theorem \ref{jj} is Theorem~\ref{bncn} below.   Indeed, Theorem~ \ref{jj} follows immediately from Theorem \ref{bncn}, the Euler characteristic interpretation of the Mobius function, the recursive definition of the Mobius function, and the well-known formula
\begin{equation}
d_n=\sum_{m=0}^{n} (-1)^m{{n} \choose {m}}(n-m)! \, .
\end{equation} 

Let  $P$  be a ranked and bounded poset of length $n$ with   minimum element $\hat 0$ and  maximum element $\hat 1$.  The maximal elements of $P^-* C_n$ are of the form $(\hat 1,j)$, for $j=0\dots, n-1$.  Let  $I_j(P)$ denote the open principal order ideal generated by $(\hat 1,j)$.  If $P$ is Cohen-Macaulay then  the homology of the order complex of $I_j(P)$ is concentrated in dimension $n-2$.

\begin{thm}
For all $j =0,\dots, n-1$, we have 
\[
\dim \rh_{n-2}(I_j(B_n))=a_{n,j},
\]
where $a_{n,j}$ is the Eulerian number indexed by $n$ and $j$; that is $a_{n,j}$ is  the number of permutations in $\sg_n$ with $j$ descents, equivalently with $j$ excedances.
\label{bncn}
\end{thm}

We have obtained two  different  proofs of Theorem \ref{bncn} both as applications of   general results on  Rees products that we derive. One of these proofs, which appears in \cite{ShWa2},  involves the theory of lexicographical shellability \cite{bj}.  The other, which is given in  Sections~\ref{treesec} and~\ref{treelemsec},
is based on the recursive definition of the M\"obius function applied to the  Rees product of $B_n$ with a poset whose Hasse diagram is a  tree.  This proof yields  a {\it $q$-analog} (Theorem~\ref{bncnq})  of Theorem~\ref{bncn}, in which the Boolean algebra $B_n$ is replaced by its $q$-analog, $B_n(q)$, the lattice of subspaces of the $n$-dimensional vector space $\F_q^n$ over the $q$ element field $\F_q$, and the Eulerian number $a_{n,j}$ is replaced by a  $q$-Eulerian number.  The proof also yields an $\mathfrak S_n$-equivariant version (Theorem~\ref{bncnsg}) of Theorem~\ref{bncn}.  The proofs of these results also appear in Sections~\ref{treesec} and~\ref{treelemsec}.
A $q$-analog and equivariant version of Theorem~\ref{jj} are derived as consequences in Section~\ref{corsec}.

Recall that the {\em major index}, $\maj(\s)$, of a permutation $\s \in \sg_n$ is  the sum of all the descents of $\s$, i.e.
$$\maj(\s) := \sum_{i: \s(i) > \s(i+1)} i,$$ and the {\em excedance number}, $\exc(\s)$,  is the number of excedances of $\s$, i.e., $$\exc(\s):= |\{i \in [n-1]: \s(i) > i\}|.$$
For $n \ge 1$, define the $q$-Eulerian polynomial $$ A^{\maj,\exc}_{n}(q,t) := \sum_{\s \in \sg_n} q^{\maj(\s)} t^{\exc(\s)}$$ and, let $A^{\maj,\exc}_{0}(q,t) = 1$.  For example, 
$$A_3^{\maj,\exc}(q,t) := 1 + (2q+q^2+q^3)t+q^2t^2.$$  For all $j$, the $q$-Eulerian number $a^{\maj,\exc}_{n,j}(q)$ is the coefficient of $t^j$ in $A_n^{\maj,\exc}(q,t)$.   The study of the $q$-Eulerian polynomials $ A^{\maj,\exc}_{n}(q,t)$ was initiated in our recent paper \cite{sw} and was subsequently further investigated in \cite{ShWa, FH,FH2,FH3,FH4}.  There are various other $q$-analogs of the Eulerian polynomials that had been extensively studied in the literature  prior to our paper; for a sample see \cite{bs,br,c,csz,foa3,fs2,fz,gar,gg,grem,gr,hag,rrw,ra,sk,st,st1,wa2}.  They involve different combinations of Mahonian and Eulerian permutation statistics, such as the major index and the descent number, the inversion index and the descent number, the inversion index and the excedance number. 

 Like $B_n^-*C_n$,  the $q$-analog
 $B_n(q)^- \ast C_n$ is
Cohen-Macaulay.  Hence  $I_{j}(B_n(q))$ has vanishing homology
below its top dimension $n-2$.  We  prove
the following $q$-analog of Theorem \ref{bncn}.

\begin{thm} \label{bncnq}
For all $j=0,1,\dots,n-1$,
\begin{equation} \label{bncnqgen}
\dim \rh_{n-2}(I_{j}(B_n(q)))= q^{{n\choose 2} + j} \, a^{\maj,\exc}_{n,j}(q^{-1}) .\end{equation}
\end{thm}

As a consequence we obtain the following $q$-analog of Theorem~\ref{jj}.
\begin{cor} \label{bncnqcor} For all $n\ge 0$, let $\mathcal D_n$ be the set of derangements in $\sg_n$.  Then
\begin{equation*} \dim \tilde H_{n-1}(B_n(q)^- \ast C_n)= \sum_{\sigma \in \mathcal D_n} q^{{n\choose 2}-\maj(\s) + \exc(\s)}.\end{equation*}
\end{cor}

The symmetric group $\mathfrak S_n$ acts on $B_n$ in an obvious way and this induces an action on $B_n^-\ast C_n$ and on each $I_j(B_n)$.   From these actions, we obtain  a representation of $\sg_n$ on $\tilde H_{n-1}(B_n^- \ast C_n)$ and on each $\rh_{n-2}(I_j(B_n))$.  We show that  these representations can be described in terms of the Eulerian quasisymmetric functions that we introduced in \cite{sw, ShWa}.  

The Eulerian quasisymmetric function $Q_{n,j}$ 
 is defined as  a sum of fundamental quasisymmetric functions associated with permutations in $\sg_n$ having $j$ excedances. The fixed-point Eulerian quasisymmetric function $Q_{n,j,k} $ refines this; it is a sum of   fundamental quasisymmetric functions associated with permutations in $\sg_n$ having $j$ excedances and $k$ fixed points.  (The precise definitions are given in Section~\ref{qfpssec}.)   Although it's not apparent from their definition,  the $Q_{n,j,k}$, and thus  the $Q_{n,j}$, are actually symmetric functions.   
  A key result of \cite{ShWa} is the following formula, which reduces to the classical formula for the exponential generating function for Eulerian polynomials,
\begin{equation}
\label{introsymgeneq}\sum_{n,j,k \geq 0}Q_{n,j,k}(\xx)t^jr^kz^n=\frac{(1-t)H(rz)}{H(zt)-tH(z)},\end{equation} where $H(z):= \sum_{n\ge 0} h_n z^n$, and $h_n$ denotes the $n$th complete homogeneous symmetric function.

Our equivariant version of Theorem~\ref{bncn} is as follows.
\begin{thm} \label{bncnsg} For all $j=0,1,\dots,n-1$,  \begin{equation} \label{iq}
{\rm {ch}} \rh_{n-2}(I_{j}(B_n) )=\omega Q_{n,j},
\end{equation}
where $\ch$ denotes the Frobenius characteristic and  $\omega$ denotes the standard involution on the ring of symmetric functions.
\end{thm}

We derive the following equivariant version of Theorem~\ref{jj} as a consequence. 
\begin{cor}\label{bncnsgcor} For all $n \ge 1$, $$ \ch \rh_{n-1}(B_n^-*C_n) = \sum_{j=0}^{n-1} \omega Q_{n,j,0}.  $$
\end{cor}

The expression on the right hand side of (\ref{introsymgeneq})  has occurred several times in the literature (see \cite[Sec.~7]{ShWa}), and these occurrences yield corollaries of Theorem~\ref{bncnsg} and Corollary~\ref{bncnsgcor}.   We discuss three of these corollaries   in Section~\ref{corsec}.  One is a consequence of a  formula of Procesi \cite{pr} and Stanley \cite{st2} on the representation of the symmetric group on the cohomology of the toric variety  associated with the Coxeter complex of $\sg_n$.  Another  corollary is a consequence of   a refinement of a result of Carlitz, Scoville and Vaughan \cite{csv} due to Stanley (cf. \cite[Theorem 7.2]{ShWa}) on words with no adjecent repeats.   The third is
 a consequence of     MacMahon's formula  \cite[Sec. III, Ch.III]{mac1} for multiset derangements.

In Section~\ref{bcsec}, we present type BC analogs (in the context of Coxeter groups) of both Theorem~\ref{jj} and its $q$-analog, Corollary~\ref{bncnqcor}.  In the type BC analog of Theorem~\ref{jj}, the Boolean algebra $B_n$ is replaced by the poset of faces of the $n$-dimensional cross polytope (whose order complex is the Coxeter complex of type BC).  The type BC derangements are the elements  of the type BC Coxeter group that have no fixed points in their action on the  vertices of the cross polytope.  
In the type BC analog of Corollary~\ref{bncnqcor},  the lattice of subspaces $B_n(q)$  is replaced by the poset of totally isotropic subspaces of $\F_q^{2n}$ (whose order complex  is the  building of type BC).  

\section{Preliminaries}  \label{prelimsec}

\subsection{Quasisymmetric functions and permutation statistics} \label{qfpssec}
In this section we review some of our work in  \cite{ShWa}.

A {\it permutation statistic} is a function $f:\bigcup_{n \geq 1}\sg_n \rightarrow \nn$.  (Here $\nn$ is the set of nonnegative integers and $\pp$ is the set of positive integers.)  Two well studied permutation statistics are the {\it excedance statistic} $\exc$  and the {\it major index} $\maj$.  For $\si \in \sg_n$, $\exc(\si)$ is the number of excedances of $\si$ and $\maj(\si)$ is the sum of all descents of $\si$, as described above. We also define the {\it fixed point statistic} $\fix(\si)$ to be the number of $i \in [n]$ satisfying $\si(i)=i$, and the {\it comajor index} $\comaj$ by 
\[
\comaj(\si):={{n} \choose {2}}-\maj(\si).
\]

\begin{remark}
Note that our definition of $\comaj$ is different from a  commonly used definition in which the comajor index of $\si \in \sg_n$ is defined to be $n\,\des(\si)-\maj(\si)$, where $\des(\s)$ is the number of descents of $\s$.
\end{remark}

For any collection $\fff_1,\ldots,\fff_r$ of permutation statistics, and any $n \in \pp$, we define the {\it generating polynomial}
\[
A_n^{\fff_1,\ldots\fff_r}(t_1,\ldots,t_r):=\sum_{\si \in \sg_n}\prod_{i=1}^{r}t_i^{\fff_i(\si)}.
\]

A {\it symmetric function} is a power series of bounded degree (with coefficients in some given ring $R$) in countably many variables $x_1,x_2,\ldots$ that is invariant under any permutation of the variables.  A {\it quasisymmetric function} is a power series $f$ in these same variables such that for any $k \in \pp$ and any three $k$-tuples $(i_1>\ldots>i_k)$, $(j_1>\ldots>j_k)$ and $(a_1,\ldots,a_k)$ from $\pp^k$, the coefficients in $f$ of $\prod_{s=1}^k x_{i_s}^{a_s}$ and $\prod_{s=1}^k x_{j_s}^{a_s}$ are equal.  Every symmetric function is a quasisymmetric function.  We write $f(\xx)$ for any power series $f(x_1,x_2,\ldots)$.

Recall that, for $n \in \nn$, the complete homogeneous symmetric function $h_n(\xx)$ is the sum of all monomials of degree $n$ in $x_1,x_2,\ldots$, and the elementary symmetric function $e_n(\xx)$ is the sum of all such monomials that are squarefree.  
The {\it Frobenius characteristic} map $\ch$ sends each virtual $\sg_n$-representation to a symmetric function (with integer coefficients) that is homogeneous of degree $n$.  There is a unique involutory automorphism $\omega$ of the ring of symmetric functions that maps $h_n(\xx)$ to $e_n(\xx)$ for every $n \in \nn$. For any representation $V$ of $S_n$, we have
\begin{equation} \label{frsg}
\omega(\ch(V))=\ch(V \otimes \sgn),
\end{equation}
where $\sgn$ is the sign representation of $\sg_n$.

For $n \in \pp$ and $S \subseteq [n-1]$, define
\[
F_{S,n}=F_{S,n}(\xx):=\sum_{\scriptsize\begin{array}{c}i_1 \geq \ldots \geq i_n \geq 1\\ j \in S \Rightarrow i_j>i_{j+1}\end{array}}x_{i_1}\dots x_{i_n}
\] and let
$F_{\emptyset,0}=1$.
Each  $F_{S,n}$ is a quasisymmetric function.  The involution $\omega$ extends to an involution on the ring of quasisymmetric functions.  In fact, $$\omega(F_{S,n}) = F_{[n-1]\setminus S,n}.$$

For $n \in \pp$, set $[\ov{n}]:=\{\ov{1},\ldots,\ov{n}\}$ and order $[n] \cup [\ov{n}]$ by
\begin{equation} \label{ovor}
\ov{1}<\ldots<\ov{n}<1<\ldots<n.
\end{equation}
For $\si=\si_1\ldots\si_n \in \sg_n$, written in one line notation, we obtain $\ov{\si}$ by replacing $\si_i$ with $\ov{\si_i}$ whenever $i$ is an excedance of $\si$.  We now define $\Exd(\si)$ to be the set of all $i \in [n-1]$ such that  $i$ is a descent of $ \ov{\s}$, i.e. the element in position $i$ of $\ov{\si}$ is larger, with respect to the order (\ref{ovor}), than that in position $i+1$.  For example, if $\si=42153$, then $\ov{\si}=\ov{4}21\ov{5}3$ and $\Exd(\si)=\{2,3\}$.

For $n \in \pp$, $0\le j < n-1$ and $0 \leq k \leq n$, we introduced in \cite{ShWa} the {\it fixed point Eulerian quasisymmetric functions}
\[
Q_{n,j,k} = Q_{n,j,k}(\x):=
\sum_{\scriptsize \begin{array}{c} \s \in \sg_n \\ \exc(\s) = j \\ \fix(\s) =k
\end{array}} F_{\Exd(\s),n}(\x),
\]
and the {\it Eulerian quasisymmetric functions}
\[
Q_{n,j}:=\sum_{k=0}^{n}Q_{n,j,k}.
\]
We also set $Q_{0,0}=Q_{0,0,0}=1$.  It turns out that the fixed point Eulerian quasisymmetric functions (and therefore the Eulerian quasisymmetric functions) are symmetric.  

We define two power series in the variable $z$ with coefficients in the ring of symmetric functions,
\[
H(z):=\sum_{n \geq 0}h_n(\xx)z^n,
\]
and
\[
E(z):=\sum_{n \geq 0}e_n(\xx)z^n.
\]
For any $n \in \pp$, the {\it $q$-number} $[n]_q$ is the polynomial $1+q+\ldots+q^{n-1}$.  The key result in \cite{ShWa} is as follows.

\begin{thm}[\cite{ShWa}, Theorem~1.2] \label{introsymgenth}
We have
\begin{eqnarray}
\label{introsymgenth1}\sum_{n,j,k \geq 0}Q_{n,j,k}(\xx)t^jr^kz^n&=&\frac{(1-t)H(rz)}{H(zt)-tH(z)}\\ \label{introsymgenth2}&=& \frac{H(rz)}{1-\sum_{n\ge 2}t[n-1]_t h_nz^n}.
\end{eqnarray}
\end{thm}

 It is shown in \cite{ShWa} that the stable principal specialization (that is, substitution of $q^{i-1}$ for each variable $x_i$) of $F_{\Exd(\s),n}$ is given by  $$F_{\Exd(\s),n}(1,q,q^2,\dots) = (q;q)_n^{-1} q^{maj(\s) -\exc(s)},$$ where $(p;q)_n:=\prod_{i=1}^n (1-pq^{i-1})$.
Hence $$\sum_{j,k \ge 0} Q_{n,j,k}(1,q,\dots)t^j r^k := (q;q)_n^{-1} A_n^{\maj,\exc,\fix} (q;q^{-1}t,r).$$
Using the stable principal specialization 
we obtained from Theorem \ref{introsymgenth} a formula for $A_n^{\maj,\exc,\fix}$.  From that formula, we derived the two following results.  Before stating them, we recall the definitions
\[
\begin{array}{ll} [n]_q!:=\prod_{j=1}^{n}[j]_q, &  \left[\begin{array}{c} n \\k\end{array}\right]_q:=\frac{[n]_q!}{[k]_q![n-k]_q!}, \\ \Exp_q(z):=\sum_{n \geq 0}q^{{n} \choose {2}}\frac{z^n}{[n]_q!}, & \exp_q(z):=\sum_{n \geq 0}\frac{z^n}{[n]_q!}. \end{array}
\]

\begin{cor}[\cite{ShWa}, Corollary 4.5] \label{comajcor}We have
\begin{equation} \label{expgeneqfixco}
\sum_{n \geq 0}A^{\comaj,\exc,\fix}_n(q,t,r)\frac{z^n}{[n]_q!}=\frac{(1-tq^{-1})\Exp_q(rz)}{\Exp_q(ztq^{-1})-(tq^{-1})\Exp_q(z)}
\end{equation}
\end{cor}

\begin{cor}[\cite{ShWa}, Corollary 4.6] \label{coexcderang} 

For all $n \ge 0$, we have
\begin{equation*} \sum_{\scriptsize \begin{array}{c} \s \in \mathfrak S_n \\ \fix(\s) = k \end{array}} q^{\comaj(\s)} t^{ \exc(\s)} = q^{k\choose 2}  \left[\begin{array}{c} n \\k\end{array}\right]_q \,\,\sum_{\s \in \mathcal D_{n-k}} q^{\comaj(\s)} t^{ \exc(\s)} .\end{equation*}
Consequently, $$ \sum_{\s \in \mathcal D_n} q^{\comaj(\s)} t^{ \exc(\s)}= \sum_{k = 0}^n (-1)^{k}  \left[\begin{array}{c} n \\k\end{array}\right]_q A_{n-k}^{\comaj,\exc}(q,t).$$
\end{cor}

\subsection{Homology of posets}

We say that a poset $P$ is {\em bounded} if it has a minimum element $\hat 0_P$ and a maximum element $\hat 1_P$.  For any poset $P$, let $\wh{P}$ be the
bounded poset obtained from $P$ by adding a minimum element  and a
maximum element  and let $P^+$ be the poset obtained from $P$ by adding only a maximum element. For a  poset $P$ with minimum element $\hz_P$,
let $P^-=P \setminus \{\hz_P\}$.  For $x \le y$ in $P$,  let $(x,y)$ denote the open interval $\{z \in P : x <z <y\}$  and $[x,y]$ denote the closed interval $\{z \in P : x \le z \le y\}$.  For $y \in P$, by {\em closed principal lower order ideal generated by} $y$, we  mean the subposet $\{x \in P : x \le y\}$.  Similarly the {\em open  principal lower order ideal generated by} $y$ is the subposet $\{x \in P : x < y\}$. Upper order ideals are defined similarly.

A poset $P$ is said to be {\em ranked} (or pure) if all its maximal chains are of the same length.  The {\em length} of a ranked poset $P$ is the common length of its maximal chains.  If $P$ is a ranked poset, the rank $r_P(y)$ of an element $y\in P$ is the length of the closed principal lower order ideal  generated by $y$.  

The {\it M\"obius invariant} of any bounded poset $P$  is given by
$$\mu(P) := \mu_P(\hat 0_P,\hat 1_P),$$ where $\mu_P$ is the M\"obius function on $P$.   It follows
from a  well known result of P. Hall (see \cite[Proposition 3.8.5]{st5}) and the Euler-Poincar\'e formula that if poset $P$ has length $n$ then
\begin{equation} \label{eupon} \mu(\hat P) =\sum_{i=0}^n (-1)^i \dim \tilde H_i(P).\end{equation}  Hence if
 $P$ is Cohen-Macaulay then for all $x \le y$ in $\hat P$
\begin{equation} \label{eupon2} \mu_P(x,y) = (-1)^{r} \dim \tilde H_r((x,y)),\end{equation}
where $r= r_P(y)-r_P(x)-2$, and if $y=x$ or $y$ covers $x$ we set $\tilde H_r((x,y)) = \C$.

Suppose a  group $G$ acts on a  poset $P$ by order preserving bijections (we say that $P$ is a $G$-poset).  The group $G$ acts simplicially on $\Delta P$ and thus arises a linear representation of $G$ on each  homology group of $P$.    Now suppose $P$ is ranked of length $n$.  The given
action also determines an action of $G$  on $P \ast X$  for
any length $n$ ranked poset $X$  defined by $(a,x)g=(ag,x)$ for all
$a \in P$, $x \in X$ and $g \in G$.  For a  ranked $G$-poset $P$ of length $n$ with a minimum element $\hat 0$, the  action of $G$  on $P$ restricts to an action on $P^-$, which gives an action of $G$ on $P^- \ast C_n$.  This action restricts to an action of $G$ on each subposet $I_j(P)$.  

We will need the following
result of Sundaram \cite{sun} (see \cite[Theorem 4.4.1]{w1}):
 If $G$ acts on a bounded  poset $P$ of length $n$  then we have the virtual $G$-module isomorphism,
\begin{equation}\label{sund}  \bigoplus_{r=0}^n (-1)^{r} \bigoplus_{x \in P/G} \tilde H_{r-2}((\hat 0,x)) \uparrow_{G_x}^G \cong_G 0,\end{equation}
where  $P/G$ denotes a complete set of orbit representatives, $G_x$ denotes the stabilizer of $x$, and $\uparrow_{G_x}^G$
denotes the induction of the $G_x$ module  from $G_x$ to $G$.
   Here $H_{r-2}((\hat 0,x))$ is the trivial representation of $G_x$ if   $x=\hat 0$ or $x$ covers $\hat 0$.

\section{Rees products with trees} \label{treesec}

We  prove the results stated in the introduction by working with the Rees product of the (nontruncated) Boolean algebra  $B_n$  with a tree and its $q$-analog, the Rees product of the (nontruncated) subspace lattice $B_n(q)$ with a tree.  Theorems~\ref{tree} and~\ref{gtreelem} will then be used to relate
these Rees products to the ones considered in the introduction.  

 For $n,t \in \PP$, let $\ttn$ be the poset whose Hasse diagram is
a complete $t$-ary tree of height $n$, with the root at the bottom.  By {\em complete}  we mean that every nonleaf node has exactly $t$ children and that all the leaves are distance $n$ from the root. The following result, which  is  interesting in its own right, will be  used to prove the results stated in the introduction.

\begin{thm}\label{treeeq123} For all $n,t\ge 1$ we have
\begin{eqnarray}\label{treeeq1}\dim \tilde H_{n-2} ((B_n * \ttn)^-) &=& tA_n(t)\\
\label{treeeq2}\dim \tilde H_{n-2} ((B_n(q) * \ttn)^-) &=&    t A^{\comaj,\exc}_n(q,qt) \\
\label{treeeq3} \ch \tilde H_{n-2} ((B_n * \ttn)^-) &=& t \sum_{j=0}^{n-1} \omega Q_{n,j}t^j.\end{eqnarray}
\end{thm}

\begin{cor}  For all $n\ge 1$ we have
\begin{eqnarray*} \dim \tilde H_{n-2} ((B_n * C_{n+1})^-) &=&n!\\
\dim  \tilde H_{n-2} ((B_n(q) * C_{n+1})^-) &=&   \sum_{\sigma \in \mathfrak S_n } q^{\comaj(\sigma) +\exc(\sigma)} \\
 \ch \tilde H_{n-2} ((B_n *C_{n+1})^-) &=& \sum_{j=0}^{n-1}  \omega Q_{n,j}.\end{eqnarray*}
\end{cor}

To prove (\ref{treeeq1}) and (\ref{treeeq2}), we make use of   two  easy Rees product results.  
A bounded ranked poset $P$ is said to be {\em uniform} if $[x,\hat 1_P] \cong [y,\hat 1_P]$ whenever $r_P(x) = r_P(y)$ (see \cite[Exercise 3.50]{st5}).  We will say that a sequence of posets $(P_0, P_1,  \dots, P_n)$ is uniform if for each $k = 0,1, \dots,n$, the poset
$P_k$ is uniform of  length $k$ and
$$  P_k \cong [x,\hat 1_{P_n}]$$ for  each $x \in P_n$ of rank $n-k$.    The sequences $(B_0,\dots,B_n)$ and
$(B_0(q),\dots,B_n(q))$ are examples of uniform sequences as are the sequences of set partition lattices $(\Pi_0,\dots, \Pi_n)$ and the sequence of face lattices  of  cross polytopes $(\wh {PCP_0},\dots,\wh {PCP_n})$.

The following result is easy to verify.
\begin{prop}  \label{uniformtree} Suppose $P$ is a uniform poset of length $n$. Then for all $t \in {\mathbb P}$,   the poset $R:=(P \ast T_{t,n})^+$ is uniform of length $n+1$.
Moreover, if $x \in P$ and $y \in R$ with $r_P(x) = r_R(y) =k$ then
 $$[y,\hat 1_R ] \cong ([x,\hat 1_P] \ast T_{t,n-k})^+.$$
\end{prop}

\begin{prop}\label{propuniform} Let $(P_0, P_1,  \dots, P_n)$ be a uniform sequence of  posets.
Then for all $t \in {\mathbb P}$,
\begin{equation} \label{uniformrec}  1+\sum_{k=0}^n W_k(P_n) [k+1]_t \mu((P_{n-k}*T_{t,n-k})^+) =0,\end{equation} where    $W_k(P)$ is the number of elements of rank $k$ in $P$.
\end{prop}

\begin{proof}
Let $R:= (P_n \ast T_{t,n})^+$ and  let $y$ have rank $k$ in $R$. By Proposition~\ref{uniformtree}, $$\mu_R(y, \hat 1_R) =\mu((P_{n-k}*T_{t,n-k})^+) . $$
Clearly $$W_k(R) = W_k(P_n) [k+1]_t$$  for all $0\le k \le n$. Hence (\ref{uniformrec}) is just the  recursive
definition of the M\"obius function applied to the dual of $R$.
\end{proof}

 To prove (\ref{treeeq1}) either take dimension in (\ref{treeeq3})  or  set $q=1$ in the proof of (\ref{treeeq2}) below.

 \begin{proof}[Proof of (\ref{treeeq2})]     We  apply Proposition~\ref{propuniform} to the uniform sequence $(B_0(q)$, $B_1(q)$, \dots, $B_n(q))$. The number of $k$-dimensional subspaces of $\F_q^n$ is given by  $$W_k(B_n(q))= \left[\begin{array}{c} n
 \\ k\end{array}\right]_q.$$  Write $\mu_n(q,t)$ for $\mu((B_n(q) \ast \ttn)^+)$. Hence by Proposition~\ref{propuniform},
  \begin{equation} \label{recbn2q}
\sum_{k=0}^n  \left[\begin{array}{c} n \\k\end{array}\right]_q [k+1]_t\mu_{n-k}(q,t)=-1.
\end{equation}

Setting
\[
F_{q,t}(z):=\sum_{j \geq 0}\mu_j(q,t)\frac{z^j}{[j]_q!}
\]
and
\[
G_{q,t}(z):=\sum_{k \geq 0}[k+1]_t\frac{z^k}{[k]_q!},
\]
we derive from (\ref{recbn2q}) that
\begin{equation} \label{gbnq}
F_{q,t}(z)=-\exp_q(z)G_{q,t}(z)^{-1}.
\end{equation}
If we assume $t >1$ we have
\begin{eqnarray*}
G_{q,t}(z) & = & \frac{1}{1-t}\sum_{k \geq 0}(1-t^{k+1})
\frac{z^k}{[k]_q!}
\\ & = & \frac{\exp_q(z)-t\exp_q(tz)}{1-t}.
\end{eqnarray*}
We calculate that
$$F_{q,t}(-z) = -(1-t) - t  \frac{(1-t)\exp_q(-tz)}{\exp_q(-z)-t\exp_q(-tz)}.$$
Using the fact that $\exp_q(-z) \Exp_q(z) = 1$, we have
$$F_{q,t}(-z) = -(1-t) - t  \frac{(1-t)\Exp_q(z)}{\Exp_q(tz)-t\Exp_q(z)}.$$
It now follows from 
Corollary~\ref{comajcor} that for all $n \ge 1$ and $t >1$,
\begin{equation}\label{mueq}\mu_n(q,t) = (-1)^{n-1} t \sum_{\sigma \in \mathfrak S_n} q^{\comaj(\sigma) + \exc(\sigma)} t^{\exc(\sigma)}.\end{equation}

One can see from (\ref{recbn2q}) and induction that  $\mu_n(q,t)$ is a polynomial in $t$.
Hence since (\ref{mueq}) holds for infinitely many integers $t$, it
holds as an identity of polynomials, which implies that it holds for $t=1$.

Since the poset $(B_n(q)*\ttn)^-$ is Cohen-Macaulay, equation (\ref{treeeq2}) holds.
  \end{proof}

  We say that a bounded ranked $G$-poset $P$ is $G$-uniform  if the following holds,
\begin{itemize}
\item  $P$ is uniform
\item $G_x \cong G_y$ for all $x,y \in P$ such that $r_P(x) = r_P(y)$
\item there is an isomorphism between $[x, \hat 1_P]$ and $[y, \hat 1_P]$ that intertwines the  actions of $G_x$ and $G_y$ for all $x,y \in P$ such that $r_P(x) = r_P(y)$.  We will write
$$ [x, \hat 1_P] \cong_{G_x,G_y}[y, \hat 1_P].$$
\end{itemize}
Given a sequence of groups $G=(G_0, G_1, \dots, G_n)$.  We say that a sequence of posets
$(P_0, P_1, \dots, P_n)$ is $G$-uniform if
\begin{itemize}
\item  $P_k$ is $G_k$-uniform of length $k$ for each $k$
\item  $G_k \cong (G_n)_x$ and $P_k \cong_{G_k,(G_n)_x} [x,\hat 1_{P_n}]$  whenever $r_{P_n}(x) = n-k$.
\end{itemize}
For example, the sequence $(B_0,B_1,\dots , B_n)$ is $(\mathfrak S_0\times \mathfrak S_n,\mathfrak S_1\times \mathfrak S_{n-1},\dots,\mathfrak S_n\times \mathfrak S_0)$-uniform, where the action of $\mathfrak S_i \times \mathfrak S_{n-i}$  on $B_i$  is given by $$(\sigma,\tau) \{a_1,\dots, a_s\} = \{\sigma(a_1),\dots, \sigma(a_s)\}$$ for $\sigma \in \mathfrak S_i ,\, \tau \in  \mathfrak S_{n-i}$ and $\{a_1,\dots, a_s\} \in B_i$.  In other words $\mathfrak S_i$ acts on subsets of $[i]$ in the usual way and $\mathfrak S_{n-i}$ acts trivially.

The following proposition is easy to verify.
\begin{prop}[Equivariant version of  Proposition \ref{uniformtree}] \label {guniformtree} Suppose $P$ is a $G$-uniform poset of length $n$. Then for all $t \in {\mathbb P}$,  the $G$-poset $R:=(P \ast T_{t,n})^+$ is $G$-uniform of length $n+1$.
Moreover, if $x \in P$ and $y \in R$ with $r_P(x) = r_R(y) =k$ then
 $$[y,\hat 1_R ] \cong_{G_y,G_x} ([x,\hat 1_P] \ast T_{t,n-k})^+.$$
\end{prop}

If $(P_0, P_1,  \dots, P_n)$ is a $(G_0,G_1,\dots,G_n)$-uniform sequence of  posets, we can view
$G_k $ as a subgroup of $G_{n}$ for each $k=0,\dots,n$. For $G$-uniform poset $P$, let $ W_k(P;G)$ be the number of $G$-orbits of  the  rank $k$ elements of $P$.
 The Lefschetz character of  a $G$-poset
$P$ of length $n\ge 0$  is defined to be the virtual representation $$L(P;G) := \bigoplus_{j=0}^{n} (-1)^j \tilde H_j(P).$$
Note that by (\ref{eupon}) the dimension of the Lefschetz character $L(P;G)$  is precisely $\mu(\hat P)$.

\begin{prop}[Equivariant version of Proposition \ref{propuniform}] \label{gpropuniform} Let $(P_0, P_1,  \dots, P_n)$ be a $(G_0,G_1,\dots,G_n)$-uniform sequence of  posets.
Then for all $t \in {\mathbb P}$,
\begin{equation} \label{guniformrec}  1_{G_n}\oplus \bigoplus_{k=0}^n W_k(P_n;G_n) [k+1]_t L((P_{n-k}*T_{t,n-k})^-;G_{n-k}) \uparrow^{G_n}_{G_{n-k}}=0.\end{equation} \end{prop}

\begin{proof}
Sundaram's equation (\ref{sund}) applied to the dual of a $G$-poset $P$ is equivalent to the following equivariant version of the recursive definition of the M\"obius function:
\begin{equation} \label{equimob} \bigoplus_{y \in P/G} L((y,\hat 1_P);G_y) \uparrow_{G_y}^G = 0, \end{equation}
where $L((y,\hat 1_P);G_y)$ is the trivial representation if $y=\hat 1_P$ and is the negative of the trivial representation if $ y$ is covered by $\hat 1_P$.  We apply (\ref{equimob}) to the $G_n$-uniform poset $R:= (P_n \ast T_{t,n})^+$.   Let $y$ have rank $k$ in $R$.
It follows from Proposition~\ref{guniformtree} that
$$L((y, \hat 1_R) ;(G_{n})_{y})\uparrow_{(G_{n})_{y}}^{G_n}\cong L((P_{n-k}*T_{t,n-k})^- ;G_{n-k})\uparrow_{G_{n-k}}^{G_n}.$$
Clearly, $$W_k(R;G_n) =  W_k(P_n;G_n) [k+1]_t $$ for all $ k $.  Thus  (\ref{guniformrec}) follows from (\ref{equimob}).
\end{proof}

   \begin{proof}[Proof of (\ref{treeeq3})] Now we apply Proposition~\ref{gpropuniform} to the  $(\mathfrak S_0\times \mathfrak S_n,\mathfrak S_1\times \mathfrak S_{n-1},\dots,\mathfrak S_n\times \mathfrak S_0)$-uniform sequence $(B_0, B_1, \dots, B_n)$. Let  $$L_n(t):= \ch \,L((B_{n}*T_{t,n})^-;\mathfrak S_{n}).$$ Clearly $W_k(B_n;\mathfrak S_n) = 1$.  Therefore by Proposition~\ref{gpropuniform},  \begin{equation} \label{grecbn2}
\sum_{k=0}^n  [k+1]_t h_k L_{n-k}(t)=-h_n.
\end{equation}

 Setting
\[
F_t(z):=\sum_{j \geq 0}L_j(t){z^j}
\]
and
\[
G_t(z):=\sum_{k \geq 0}[k+1]_t h_k {z^k}
\]
we derive from (\ref{grecbn2}) that
\begin{equation} \label{gbn}
F_t(z)G_t(z) =- H(z).
\end{equation}
Now if $t >1$,
\begin{eqnarray*}
G_t(z) & = & \frac{1}{1-t}\sum_{k \geq 0}(1-t^{k+1}) h_k
{z^k}
\\ & = & \frac{H(z)-tH(tz)}{1-t},
\end{eqnarray*}
and we thus have
\begin{equation} \label{gqbn}
F_t(z)=-\frac{(1-t)H(z)}{H(z)-tH(tz)}.
\end{equation}
We calculate that
\begin{equation} \label{gqp}
F_t(-z)=-(1-t) - t\frac{(1-t)H(-tz)}{H(-z) - t H(-tz)}.
\end{equation}
 Using the fact that $H(-z)E(z) = 1$ we have
$$F_t(-z)=-(1-t) - t\frac{(1-t)E(z)}{E(tz)-tE(z)}.$$
By applying the standard symmetric function involution $\omega$, we obtain
$$ \omega F_t(-z) = -(1-t) - t\frac{(1-t)H(z)}{H(tz)-tH(z)}.$$
It follows from this and Theorem~\ref{introsymgenth} that for all $n \ge 1$ and $t >1$,
\begin{equation} \label{gnewpr}
\omega L_n(t)=(-1)^{n-1}t \sum_{j=0}^{n-1} Q_{n,j} t^j
\end{equation}

By (\ref{grecbn2}) and induction, $L_n(t)$ is a polynomial in $t$.  Hence (\ref{gnewpr}) holds for $t=1$ as well.
Since $(B_n*\ttn)^-$ is Cohen-Macaulay we are done.
   \end{proof}

\section{The tree lemma}  \label{treelemsec}
The following result and Theorem~\ref{treeeq123} are all that is needed to  prove  Theorems~\ref{bncn} and~\ref{bncnq}, since $B_n$ and $B_n(q)$ are self-dual and Cohen-Macaulay.

\begin{thm}[Tree Lemma] \label{tree}
Let $P$ be a bounded, ranked poset of length $n$.  Then for all $t
\in {\mathbb P}$,
\begin{equation} \label{treeeq}
\sum_{j=1}^{n}\mu(\wh{I_{j-1}(P)}) t^j = - \mu((P^**\ttn)^+),
\end{equation}
where $P^\ast$ is the dual of $P$.
\end{thm}

Before we can prove Theorem~\ref{tree}, we need a few lemmas.
  Set
\[
R(P):=P \ast\{x_0<x_1<\ldots <x_n\}
\]
and for $i \in [n]$, let $R_i(P)$ be the closed lower order ideal in $R(P)$ generated
by $(\hat 1_P,x_i)$.
Set
\[
R_i^+(P):=\{(a,x_j) \in R_i(P):j>0\}
\]
and
\[
R_i^-(P):=R_i(P) \setminus R_i^+(P).
\]

\begin{lemma}
The posets $R_i^+(P)$ and $I_{i-1}(P)^+ $ are
isomorphic. \label{qi}
\end{lemma}

\begin{proof}
The map that sends $(a,x_j)$ to $(a,j-1)$ is an isomorphism.
\end{proof}

  An {\it antiisomorphism} from poset $X$ to a poset $Y$ is
an isomorphism $\psi$ from $X$ to $Y^*$.  In other words, $\psi$
is an order reversing bijection from $X$ to $Y$ with order
reversing inverse.

\begin{lemma}
For $0 \leq i \leq n$, the map $\psi_i:R_i(P) \rightarrow
R_i(P^\ast)$ given by $\psi_i((a,x_j))=(a,x_{i-j})$ is an
antiisomorphism. \label{antii}
\end{lemma}

\begin{proof}
We show first that $\psi_i$ is well-defined, that is, if $(a,x_j)
\in R_i(P)$ then $(a,x_{i-j}) \in R_i(P^\ast)$.  For $a \in P$ and
$j \in \{0,\ldots,n\}$ we have $(a,x_j) \in R_i(P)$ if and only if
the three conditions
\begin{itemize}
\item[(1)] $0 \leq j \leq i$ \item[(2)] $r_P(a) \geq j$
\item[(3)] $n-r_P(a) \geq i-j$
\end{itemize}
hold.  If (1), (2), (3) hold then so do all of
\begin{itemize}
\item[($1^\prime$)] $0 \leq i-j \leq i$ \item[($2^\prime$)]
$r_{P^\ast}(a)=n-r_P(a) \geq i-j$  \item[($3^\prime$)]
$n-r_{P^\ast}(a)=r_P(a) \geq j=i-(i-j)$,
\end{itemize}
and ($1^\prime$), ($2^\prime$), ($3^\prime$) together imply that $(a,x_{i-j}) \in
R_i(P^\ast)$.  The map $\psi^\ast_i:R_i(P^\ast)\rightarrow R_i(P)$
given by $\psi^\ast_i((a,x_j))=(a,x_{i-j})$ is also well-defined by
the argument just given, and $\psi^\ast_i=\psi_i^{-1}$, so
$\psi_i$ is a bijection.

Now for $(a,x_j)$ and $(b,x_k)$ in
$R_i(P)$, we have $(a,x_j)<(b,x_k)$ if and only if the three
conditions
\begin{itemize}
\item [(4)] $a \leq_P b$ \item[(5)] $j \leq k$  \item[(6)]
$r_P(b)-r_P(a) \geq k-j$
\end{itemize}
hold.  If (4), (5), (6) hold then so do all of
\begin{itemize}
\item[($4^\prime$)] $b \leq_{P^\ast} a$ \item[($5^\prime$)] $i-k \leq i-j$ \item[($6^\prime$)] $r_{P^\ast}(a)-r_{P^\ast}(b)=r_P(b)-r_P(a) \geq
k-j=(i-j)-(i-k)$,
\end{itemize}
and ($4^\prime$), ($5^\prime$), ($6^\prime$) together imply that in $R_i(P^\ast)$ we have
$(b,x_{i-k}) \leq (a,x_{i-j})$.  Therefore, $\psi_i$ is order
reversing, and the same argument shows that $\psi_i^\ast$ is order
reversing.
\end{proof}

\begin{cor}\label{sumcor}
 For $1 \leq i \leq n$ we have
\begin{equation} \label{sumeq}
\mu(\wh{I_{i-1}(P)})=\sum_{(a,x_i) \in
R_i(P^\ast)}\mu_{R_i(P^\ast)}((\hat 1_P,x_0),(a,x_i)).
\end{equation}
\end{cor}

In case the notation has confused the reader, we remark before
proving Corollary \ref{sumcor} that the sum on the right side of
equality (\ref{sumeq}) is taken over all pairs $(a,x_i)$ such that
$a \in P$ with $r_P(a) \leq n-i$ (so $r_{P^\ast}(a) \geq i$),
and that $\hat 1_P$, being the maximum element of $P$, is the
minimum element of $P^\ast$ (so $(\hat 1_P,x_0)$ is the minimum
element of $R_i(P^\ast)$).

\begin{proof}
We have
\begin{eqnarray*}
\mu(\wh{I_{i-1}(P)})& = & -\sum_{\alpha \in I_{i-1}(P)^+}\mu_{\wh{I_{i-1}(P)}}(\alpha,(\hat 1_P,i-1)) \\ & = &
-\sum_{\beta \in R_i^+(P)}\mu_{R_i^+(P)}(\beta,(\hat 1_P,x_i)) \\ &
= & \sum_{\gamma=(a,x_0) \in
R_i^-(P)}\mu_{R_i(P)}(\gamma,(\hat 1_P,x_i)) \\ & = &
\sum_{\gamma=(a,x_0) \in R_i^-(P)}
\mu_{R_i(P^\ast)}(\psi_i((\hat 1_P,x_i)),\psi_i(\gamma)) \\ & = &
\sum_{\gamma=(a,x_0) \in
R_i^-(P)}\mu_{R_i(P^\ast)}((\hat 1_P,x_0),(a,x_i)) \\ & = &
\sum_{(a,x_i) \in
R_i(P^\ast)}\mu_{R_i(P^\ast)}((\hat 1_P,x_0),(a,x_i)).
\end{eqnarray*}
Indeed, the first equality follows from the definition of the
M\"obius function; the second follows from Lemma \ref{qi}; the
third follows from the definition of the M\"obius function and the
fact that $\mu_{R_i^+(P)}$ is the restriction of $\mu_{R_i(P)}$ to
$R_i^+(P) \times R_i^+(P)$ (as $R_i^+(P)$ is an upper order ideal in
$R_i(P)$); the fourth follows from Lemma \ref{antii} and the last
two follow from the definition of $\psi_i$.
\end{proof}

\begin{proof}[Proof of Tree Lemma (Theorem~\ref{tree})] The poset $\ttn$ has  exactly $t^j$ elements of rank $j$ for each $j=0,\dots, n$.  Let $r_T$ be the rank function of $\ttn$ and let $\hat 0_T$ be the minimum element of $\ttn$.

We have \begin{eqnarray*}  \mu((P^**\ttn)^+)
 & = &
-\sum_{\alpha \in P^\ast \ast \ttn}\mu_{P^\ast \ast
\ttn}((\hat 1_P,\hat 0_T),\alpha) \\ & = & -\sum_{j=0}^{n} \sum_{\alpha
\in P^\ast_{n,t,j}}\mu_{P^\ast \ast \ttn}((\hat 1_P,\hat 0_T),\alpha),
\end{eqnarray*}
where
\[
P^\ast_{n,t,j}:=\{(a,w) \in P^\ast \ast \ttn: r_T(w)=j\}.
\]

We have
\begin{eqnarray*}
\sum_{\alpha \in P^\ast_{n,t,0}}\mu_{P^\ast \ast
\ttn}((\hat 1_P,\hat 0_T),\alpha) & = & \sum_{a \in P^\ast}\mu_{P^\ast
\ast \ttn}((\hat 1_P,\hat  0_T),(a,\hat 0_T)) \\ & = & \sum_{a \in
P^\ast}\mu_{P^\ast}(\hat 1_P,a) \\ & = & 0.
\end{eqnarray*}

Now fix $j \in [n]$.  For any $w \in \ttn$ with $r_T(w)=j$, the
interval $[\hat 0_T,w]$ in $\ttn$ is a chain of length $j$.  Therefore,
for any $(a,w) \in P^\ast_{n,t,j}$, the interval
$[(\hat 1_P,\hat 0_T),(a,w)]$  in $P^\ast \ast \ttn$ is isomorphic with
the interval $[(\hat 1_P,x_0),(a,x_j)]$ in $R_j(P^\ast)$.  For any $a \in
P^\ast$, the four conditions
\begin{itemize} \item $r_{P^\ast}(a) \geq j$, \item $(a,w) \in
P^\ast_{n,t,j}$ for some $w \in \ttn$, \item $(a,v) \in
P^\ast_{n,t,j}$ for every $v \in \ttn$ satisfying $r_T(v)=j$,
\item $(a,x_j) \in R_j(P^\ast)$
\end{itemize}
are all equivalent.  There are exactly $t^j$ elements $v \in \ttn$
 of rank $j$. It follows that
\[
\sum_{\alpha \in P^\ast_{n,t,j}}\mu_{P^\ast \ast
\ttn}((\hat 1_P,\hat 0 _T),\alpha)=t^j\sum_{(a,x_j) \in
R_j(P^\ast)}\mu_{R_j(P^\ast)}((\hat 1_P,x_0),(a,x_j)),
\]
and the Tree Lemma now follows from Corollary \ref{sumcor}.
\end{proof}

Since $B_n$ is Cohen-Macaulay and self-dual, the following result shows that Theorem \ref{bncnsg} is equivalent to (\ref{treeeq3}).

\begin{thm}[Equivariant Tree Lemma] \label{gtreelem} Let $P$ be a bounded, ranked $G$-poset of length $n$.
Then for all $t \in {\mathbb P}$,
\begin{equation} \label{eqtreeeq}
\bigoplus_{j=1}^{n}t^j L(I_{j-1}(P);G) \cong_G - L((P^**\ttn)^-;G).
\end{equation}
Consequently, if $P$ is Cohen-Macaulay then for all $t \in \PP$,
$$\bigoplus_{j=1}^n  t^j \tilde H_{n-2}(I_{j-1}(P)) \cong_G  \tilde H_{n-1}((P^**\ttn)^-).$$
\end{thm}

\begin{proof}  The proof is an equivariant version of the proof of the Tree Lemma.
In particular,  the  isomorphism of Lemma~\ref{qi} is $G$-equivariant,
as is the antiisomorphism of Lemma~\ref{antii}.

  The equivariant version of (\ref{sumeq}) is
\begin{equation} \label{equisumeq}  L(I_{i-1}(P);G) = \bigoplus_{(a,x_i) \in R_i(P^\ast)/G} L(((\hat 1_P,x_0),(a,x_i));G_a)\uparrow_{G_a}^G.\end{equation}
 To prove (\ref{equisumeq}) we let (\ref{equimob}) play the role of the recursive definition of M\"obius function in the proof of (\ref{sumeq}).

 To prove
(\ref{eqtreeeq}) we follow the proof of the Tree Lemma again letting (\ref{equimob}) play the role of the recursive definition of M\"obius function, and in the last step applying  (\ref{equisumeq}) instead of (\ref{sumeq}).
\end{proof}

\section{Corollaries} \label{corsec}
In this section we restate and prove  Corollaries~\ref{bncnqcor} and~\ref{bncnsgcor} and discuss some other corollaries that were mentioned in the introduction.

\begin{cor}[to Theorem~\ref{bncnq}] \label{bncnqcor2} For all $n\ge 0$, let $\mathcal D_n$ be the set of derangements in $\sg_n$.  Then
\begin{equation*} \dim \tilde H_{n-1}(B_n(q)^- \ast C_n)= \sum_{\sigma \in \mathcal D_n} q^{\comaj(\s) + \exc(\s)}.\end{equation*}
\end{cor}

\begin{proof}   Since $B_n(q)^- \ast C_n$ is
Cohen-Macaulay and the number of $m$-dimensional subspaces of $\F_q^n$ is  $ \left[\begin{array}{c} n \\m\end{array}\right]_q $, the M\"obius function recurrence for $(B_n(q)^- \ast C_n) \cup \{\hat 0, \hat 1\}$ is equivalent to 
$$ \dim \tilde H_{n-1}(B_n(q)^- \ast C_n) = \sum_{m=0}^n  \left[\begin{array}{c} n \\m\end{array}\right]_q(-1)^{n-m} \sum_{j=0}^{m-1} \dim \tilde H_{m-2}(I_j(B_m(q))).$$
It therefore follows from Theorem~\ref{bncnq} that
$$\dim \tilde H_{n-1}(B_n(q)^- \ast C_n) = \sum_{m=0}^n  \left[\begin{array}{c} n \\m\end{array}\right]_q(-1)^{n-m} \sum_{\sigma \in \mathfrak S_m} q^{\comaj(\s) + \exc(\s)}.$$
The result thus follows from Corollary~\ref{coexcderang} \end{proof}

\begin{cor}[to Theorem~\ref{bncnsg}] \label{bncnsgcor2}We have
\begin{equation} \label{bncnsgeq2}\sum_{n\ge 0} \ch \rh_{n-1}(B_n^-*C_n)  z^n = \frac {1} {1-\sum_{i\ge 2} (i-1) e_iz^i }.\end{equation}
Equivalently, \begin{equation} \label{bncnsgeq3} \ch \rh_{n-1}(B_n^-*C_n) = \sum_{j=0}^{n-1} \omega Q_{n,j,0}.  \end{equation}
\end{cor} 

\begin{proof}    Applying (\ref{sund}) to the Cohen-Macaulay $\sg_n$-poset  $\wh{B_n^- * C_n}$, we have
$$\tilde H_{n-1}(B_n^- * C_n) \cong_{\sg_n}\bigoplus_{m=0}^n (-1)^{n-m} \bigoplus_{j=0}^{m-1}  \left(\tilde H_{m-2}(I_j(B_m)) \otimes 1_{\sg_{n-m}}\right)\uparrow_{\sg_m\times\sg_{n-m}}^{\sg_n},$$
where $1_G$ denotes the trivial representation of a group $G$.  From this we obtain
\begin{equation} \label{sundeq2} \ch \tilde H_{n-1}(B_n^- * C_n)= \sum_{m=0}^n (-1)^{n-m} \sum_{j=0}^{m-1} \ch \tilde H_{m-2}(I_j(B_m)) h_{n-m}
\end{equation}
Hence
\begin{eqnarray*} \sum_{n \ge 0} \ch \tilde H_{n-1}(B_n^- \ast C_n)\, z^n &=& H(-z)\sum_{n\ge 0}z^n\sum_{j =0}^{n-1} \ch \tilde H_{n-2}(I_j(B_n)).\end{eqnarray*}

It follows from Theorem~\ref{bncnsg} and (\ref{introsymgenth2}) that
$$\sum_{n\ge 0}z^n\sum_{j =0}^{n-1} \ch \tilde H_{n-2}(I_j(B_n))t^j = \frac{E(z)}{1-\sum_{n\ge 2} t[n-1]_t e_nz^n}.$$
By setting $t=1$ and using the fact that  $E(z) H(-z) =1$,  we obtain (\ref{bncnsgeq2}). 
Equation~(\ref{bncnsgeq3})  follows from (\ref{bncnsgeq2}) and (\ref{introsymgenth2}). \end{proof}

We now present some additional corollaries of Theorem~\ref{bncnsg} and Corollary~\ref{bncnsgcor}, which follow from the occurrence of the right hand side of (\ref{introsymgeneq}) in  various results in the literature.

Let $X_n$ be the toric variety naturally associated to the Coxeter complex $C_n$ for the reflection group $\sg_n$.  (See, for example \cite{Br}, for a
discussion of Coxeter complexes and \cite{fu} for an explanation
of how toric varieties are associated to polytopes.) The action of $\sg_n$ on $C_n$ induces an action on $X_n$ and thus a representation on each cohomology group of $X_n$.  Now $X_n$ is a complex manifold of dimension $n-1$, and can have nontrivial homology only in dimensions $2j$, for $0 \leq j \leq n-1$.  Using work of Procesi \cite{pr}, Stanley shows in \cite{st2} that
\[
\sum_{n\ge 0} \sum_{j=0}^{n-1} \ch H^{2j}(X_n)\,t^{j} z^n
= {(1-t) H(z) \over H(zt) -tH(z)}.
\]  

Combining this with Theorem~\ref{bncnsg} and equations (\ref{introsymgeneq}) and (\ref{frsg}), we obtain the following result.

\begin{cor}[to Theorem~\ref{bncnsg}] For all $j =0,\dots,n-1$, we have the following isomorphism of $\sg_n$-modules
$$ \rh_{n-2}(I_j(B_n)) \cong_{\sg_n} H^{2j}(X_n) \otimes \sgn.$$
\end{cor}

It would be interesting to find a topological explanation for this isomorphism, in particular one that extends the isomorphism to other Coxeter groups.  

Another corollary is an immediate consequence of  a refinement of a result of Carlitz, Scoville and Vaughan \cite{csv} due to Stanley (cf. \cite[Theorem 7.2]{ShWa}).
\begin{cor}For all $j =0,\dots,n-1$, let $W_{n,j}$ be the set of all words of length $n$ over the alphabet of positive integers with the properties that  no adjacent letters are equal and there are exactly $j$ descents .  Then
$$\ch  \rh_{n-2}(I_j(B_n)) = \sum_{w:=w_1\cdots w_n \in W_{n,j}} x_{w_1}x_{w_2}\cdots x_{w_n}. $$ \end{cor}

The following equivariant version of Theorem~\ref{jj} is an immediate consequence of  Corollary~\ref{bncnsgcor2} and   MacMahon's formula  \cite[Sec. III, Ch.III]{mac1} for multiset derangements.   A {\it multiset derangement} of order $n$ is a $2 \times n$ matrix $D=(d_{i,j})$ of positive integers such that

\begin{itemize}
\item $d_{1,j} \leq d_{1,j+1}$ for all $j \in [n-1]$,

\item the multisets $\{d_{1,j}:j \in [n]\}$ and $\{d_{2,j}:j \in [n]\}$ are equal, and

\item $d_{1,j} \neq d_{2,j}$ for all $j \in [n]$.

\end{itemize}

Given a multiset derangement $D$, we write $x^D$ for $\prod_{j=1}^{n}x_{1,j}$.

\begin{cor}[to Corollary~\ref{bncnsgcor2}]
For all  $n \ge 1$, we have
\begin{equation} \label{derange1} \ch \tilde H_{n-1}(B_n^-*C_n) = \sum_{D\in \mathcal {MD}_{n}} {\rm x}^D,
\end{equation}
where $\mathcal {MD}_{n,}$ is the set of all multiset derangements of order $n$.
\end{cor}

\section{Type BC-analogs} \label{bcsec} In this section we present type BC analogs (in the context of Coxeter groups) of both the 
Bj\"orner-Welker-Jonsson derangement result (Theorem~\ref{jj}) and its q-analog  (Corollary~\ref{bncnqcor}).

A   poset $P$ with a $\hat 0_P$ is said to be a {\em simplicial poset} if $[\hat 0_P, x]$ is a Boolean
algebra for all $x \in P$.  The prototypical example of a simplicial poset is the poset of faces  of a simplicial complex.   In fact, every simplicial poset is isomorphic to the face poset of some regular CW complex (see \cite{bj2}).
The next result follows immediately from Theorem~\ref{bncn} and
the definition of the M\"obius function.  For a ranked poset $P$ of length $n$ and $r \in \{0,1,\dots, n\}$, let $W_r(P)$ be the $r${\em th} Whitney number of the second kind of $P$, that is the number elements of rank $r$ in $P$.

\begin{cor}[of Theorem~\ref{bncn}]
Let $P$ be a ranked simplicial poset of length $n$.     Then
$$
\mu(\wh{P^- \ast C_n})=\sum_{r=0}^{n}(-1)^{r-1}W_r(P)r!.
$$
\label{boopo}
\end{cor}

We think of $B_n$ as
the poset of faces of a $(n-1)$-simplex, whose barycentric
subdivision is the Coxeter complex of type $A$.  Then $d_n$
is the number of derangements in the action of the associated
Coxeter group $\mathfrak S_n$ on the vertices of the simplex.  Let $PCP_n$
be the poset of  {\it simplicial} (that is, proper) faces
of the $n$-dimensional crosspolytope $CP_n$ (see for example
\cite[Section 2.3]{blswz}), whose barycentric subdivision is the Coxeter complex
of type BC. The associated Weyl group, which is isomorphic to the
wreath product $\mathfrak S_n[\zz_2]$, acts by reflections on $CP_n$ and
therefore on its vertex set.  Let $d_n^{BC}$ be the number of
derangements in this action on vertices.

\begin{thm}
For all $n$, we have
\[\dim \tilde H_{n-1}({PCP^{-}_n \ast C_n}) =
d_n^{BC}.
\]
\label{typeb}
\end{thm}

\begin{proof}
It is well known and straightforward to prove by induction on $n$
that, for $0 \leq r \leq n$, the number of $(r-1)$-dimensional faces
of $CP_n$ is $2^r{{n} \choose {r}}$.  Corollary \ref{boopo} gives
\[
 \mu(\wh{PCP^{-}_n \ast C_n})=\sum_{r=0}^{n}(-1)^{r-1}2^r{{n} \choose {r}}r!.
\]  Hence since $PCP^{-}_n$ is Cohen-Macaulay,  we have,
$$\dim \tilde H_{n-1}({PCP^{-}_n \ast C_n}) = \sum_{r=0}^{n}(-1)^{n-r}2^r{{n} \choose {r}}r!.$$

On the other hand, we may identify the vertices of $CP_n$ with
elements of $[n] \cup [\ov{n}]$, where $ [\ov{n}] =\{\bar 1,\dots,
\bar n\}$, so that the action of the Weyl group $W \cong
\S_n[\zz_2]$ is determined by the following facts.
\begin{itemize}
\item Each element $w \in W$ can be written uniquely as $w=(\sigma,v)$ with
$\sigma \in \S_n$ and $v \in \zz_2^n$. \item Any element of the
form $(\sigma,0)$ maps $i \in [n]$ to $\sigma(i)$ and $\ov{i} \in
[\ov{n}]$ to $\ov{\sigma(i)}$. \item Any element of the form
$(1,e_i)$, where $e_i$ is the $i^{th}$ standard basis vector in
$\zz_2^n$, exchanges $i$ and $\ov{i}$, and fixes all other
vertices.
\end{itemize}
It follows that for each $S \subseteq [n]$, the pointwise
stabilizer of $S$ in $W$ is exactly the pointwise stabilizer of
$\ov{S}:=\{\ov{i}:i \in S\}$ and is isomorphic to
$\S_{n-|S|}[\zz_2]$.  Using inclusion-exclusion as is done to
calculate $d_n$, we get
\[
d_n^{BC}=\sum_{j=0}^{n}(-1)^j{{n} \choose {j}}2^{n-j}(n-j)!.
\]
\end{proof}

Muldoon and Readdy \cite{mr} have recently obtained a dual version of Theorem~\ref{typeb} in which
the Rees product of the dual of $PCP_n$ with the chain is considered.  

Next we consider a poset that can be viewed as both a $q$-analog of $PCP_n$ and a type $BC$ analog of $B_n(q)$.  Let $\langle \cdot,\cdot\rangle$ be a nondegenerate, alternating bilinear form on the vector space $\F_q^{2n}$.  A subspace $U$ of $\F_q^{2n}$ is said to be {\em totally isotropic} if $\langle u,v\rangle = 0$ for all $u,v \in U$.   Let $PCP_n(q)$ be the poset of totally isotropic subspaces of  $\F_q^{2n}$.  The order complex of $PCP_n(q)$ is the building of type $BC$, naturally associated to a finite group of Lie type $B$ or $C$ (see for example \cite[Chapter V]{Br}, \cite[Appendix 6]{Ro}).  Thus we have both a $q$-analog of $PCP_n$ and a type BC analog of $B_n(q)$ (since the order complex of $B_n(q)$ is the building of type $A$).

Clearly $PCP_n(q)$ is a lower order ideal of $B_{2n}(q)$.

\begin{prop} \label{wpcpq} The maximal elements of  $PCP_n(q)$ all have dimension $n$.  For $r=0,\dots, n$, the number of $r$-dimensional isotropic subspaces of   $\F_q^{2n}$ is given by $$W_r(PCP_n(q))=  \left[\begin{array}{c} n \\r\end{array}\right]_q (q^n+1)(q^{n-1}+1) \cdots(q^{n-r+1}+1).$$
\end{prop}

\begin{proof}

The first claim of the proposition is a well known fact (see for example \cite[Chapter 1]{Ro}).  The second claim is also a known fact; we sketch a proof here.  The number of ordered bases for any $k$-dimensional subspace of $\F_q^{2n}$ is
\[
\prod_{j=0}^{k-1}(q^k-q^j).
\]

On the other hand, we can produce an ordered basis for a $k$-dimensional totally isotropic subspace of $\F_q^{2n}$ in $k$ steps, at each step $i$ choosing $v_i \in \langle v_1,\ldots,v_{i-1}\rangle^\perp \setminus \langle v_1,\ldots,v_{i-1} \rangle$.  The number of ways to do this is
\[
\prod_{j=0}^{k-1}(q^{2n-j}-q^j),
\]
and the proof is completed by division and manipulation.
\end{proof}

It was shown by  Solomon  \cite{So} that $PCP_n(q)$ is Cohen-Macaulay.  Hence  so is  ${PCP_n(q)^-*C_n}$.  We will show that the dimension of $\tilde H_{n-1} ({PCP_n(q)^-*C_n})$ is a polynomial in $q$ with nonnegative integral coefficients and give a combinatorial interpretation of the coefficients. 
 We first need the following q-analog of Corollary~\ref{boopo}.
We  say that a poset $P$ with $\hat 0_P$ is $q$-simplicial if each interval $[\hat 0_P,x]$ is isomorphic to $B_j(q)$ for some $j$.  

\begin{cor}[of Theorem~\ref{bncnq}]  \label{qboopo} Let $P$ be a ranked $q$-simplicial poset of length $n$.  Then $$
\mu(\wh{P^- \ast C_n})=\sum_{r=0}^{n}(-1)^{r-1}W_r(P)\sum_{\s \in \S_r} q^{\comaj(\s)+\exc(\s)}.
$$
\end{cor}

\begin{thm}\label{qbth} For all $n\ge 0$, let  $d_n(q) := \sum_{\sigma \in \mathcal D_n} q^{\comaj(\sigma)+\exc(\sigma)}$.  Then \begin{equation} \label{qbeq}  \dim \tilde H_{n-1} ({PCP_n(q)^-*C_n}) = \sum_{k =0}^n \left[\begin{array}{c} n \\k\end{array}\right]_q \,\,q^{k^2}\,\,\prod_{i=k+1}^n (1+q^i) \,\, d_{n-k}(q). \end{equation}  Consequently,
$\dim \tilde H_{n-1} ({PCP_n(q)^-*C_n})$ is a polynomial in $q$ with nonnegative integer coefficients.
\end{thm}

\begin{proof}  We have by Proposition~\ref{wpcpq}, Corollary~\ref{qboopo}, and the fact that   ${PCP_n(q)^-*C_n}$ is Cohen-Macaulay, $$\dim \tilde H_{n-1} ({PCP_n(q)^-*C_n}) = \sum_{j=0}^n (-1)^j \left[\begin{array}{c} n \\j\end{array}\right]_q\,\, \prod_{i=j+1}^n (1+q^i)\,\, a_{n-j}(q),$$ where $a_n(q):= \sum_{\sigma \in \mathfrak S_n} q^{\comaj(\sigma)+\exc(\sigma)}$.
On the other hand by Corollary~\ref{coexcderang}, the right hand side of
(\ref{qbeq}) equals
$$\sum_{k =0}^n \left[\begin{array}{c} n \\k\end{array}\right]_q \,\,q^{k^2}\,\,\prod_{i=k+1}^n (1+q^i)  \sum_{m=0}^{n-k} (-1)^m \left[\begin{array}{c} n-k \\m\end{array}\right]_q a_{n-k-m}(q)$$

\begin{eqnarray*} &=&  \sum_{j\ge 0} a_{n-j}(q) \sum_{k \ge 0} \left[\begin{array}{c} n \\k\end{array}\right]_q
 \,\,q^{k^2}\,\,\prod_{i=k+1}^n (1+q^i) (-1)^{j-k}  \left[\begin{array}{c} n-k \\j-k\end{array}\right]_q \\ &=&  \sum_{j\ge 0} a_{n-j}(q)  \left[\begin{array}{c} n \\j\end{array}\right]_q\sum_{k \ge 0} \left[\begin{array}{c} j \\k\end{array}\right]_q
 \,\,q^{k^2}\,\,\prod_{i=k+1}^n (1+q^i) (-1)^{j-k} .\end{eqnarray*}
 Thus to prove (\ref{qbeq}) we need only show that
 $$ \prod_{i=j+1}^n (1+q^i) = \sum_{k \ge 0} \left[\begin{array}{c} j \\k\end{array}\right]_q
 \,\,q^{k^2}\,\,\prod_{i=k+1}^n (1+q^i) (-1)^{k} ,$$ holds for all $n$ and $j$.
 By Gaussian inversion this is equivalent to,
 $$q^{j^2}(-1)^j \prod_{i=j+1}^n(1+q^i) = \sum_{k \ge 0} \left[\begin{array}{c} j \\k\end{array}\right]_q (-1)^{j-k} q^{j-k \choose 2} 
\prod_{i=k+1}^n (1+q^i) ,$$
which is in turn equivalent to,
\begin{equation} \label{neweq} q^{j^2}(-1)^j  = \sum_{k \ge 0} \left[\begin{array}{c} j \\k\end{array}\right]_q (-1)^{j-k} q^{j-k \choose 2} 
\prod_{i=k+1}^j (1+q^i)  .\end{equation}

To prove (\ref{neweq}) we use the q-binomial formula,
$$\prod_{i=0}^{n-1} (x+yq^i) = \sum_{k\ge 0} \left[\begin{array}{c} n \\k\end{array}\right]_q q^{k\choose 2} x^{n-k}y^k.$$  Set $y=1$ and use Gaussian inversion to obtain
$$x^n = \sum_{k\ge 0} \left[\begin{array}{c} n \\k\end{array}\right]_q (-1)^{n-k} \prod_{i=0}^{k-1} (x+q^i)$$
Now set $x=q^n$ to obtain
\begin{eqnarray*}q^{n^2} &=& \sum_{k\ge 0} \left[\begin{array}{c} n \\k\end{array}\right]_q (-1)^{n-k} \prod_{i=0}^{k-1} (q^n+q^i)\\ 
&=&  \sum_{k\ge 0} \left[\begin{array}{c} n \\k\end{array}\right]_q (-1)^{n-k} q^{k\choose 2}\prod_{i=0}^{k-1} (q^{n-i}+1).\end{eqnarray*}
\end{proof}

Using the standard identification of  elements of $\mathfrak S_n[\Z_2]$ with barred permutations (i.e., permutations  written in one line notation with some subset  of the letters  barred), the derangements of Theorem~\ref{typeb} are the barred permutations $\s=\s_1\cdots \s_n$ for which $\s_i \ne i$ for all $i \in [n]$.  Let $\mathcal D^{BC}_n$ be the set of such barred permutations.  If $\s$ is a barred permutation, let $|\s|$ be the ordinary permutation obtained by removing the bars from $\s$.
 For $\sigma \in \ \mathcal D_n^{BC}$, let $\tilde \sigma$ be the word obtained by rearranging the letters of $\sigma$ so that the fixed points of $|\sigma|$, which are all barred in $\sigma$, come first in increasing order with bars intact, followed by subword of nonfixed points  of $|\sigma|$ also with bars intact.  Now let $S$ be the set of positions in which bars appear in $\tilde\sigma$.  Define the {\em bar index}, $\bnd(\s)$ of $\sigma$ to be $ \sum_{i\in S} i$.  For example if $\s =\bar 3 \bar 2 5 \bar 4 \bar 6 1 \bar 7$ then $\tilde \s = \bar 2 \bar 4 \bar 7 \bar 3 5 \bar 6 1$ and so $\bnd (\s) = 1+2+3+4+6$.

\begin{cor} $$ \dim \tilde H_{n-1} ({PCP_n(q)^-*C_n}) = \sum_{\sigma \in \mathcal D_n^{BC}} q^{\comaj(|\sigma|) + \exc(|\sigma|) + \bnd(\sigma)}$$
\end{cor}

\begin{proof}   By Corollary~\ref{coexcderang} we have,
$$ \sum_{\sigma \in \mathcal D_n^{BC}} q^{\comaj(|\sigma|) + \exc(|\sigma|)} p^{ \bnd(\sigma)}\hspace{2in}$$
\begin{eqnarray*} &=& \sum_{k=0}^n  \sum_{\scriptsize \begin{array}{c} \s \in \mathfrak S_n \\ \fix(\s) = k\end{array}} q^{\comaj(\s)+\exc(\s)} p^{k+1 \choose 2} \prod_{i=k+1}^n (1+p^i) \\&=& \sum_{k=0}^n  \left[\begin{array}{c} n \\k\end{array}\right]_q q^{k\choose 2} d_{n-k}(q) p^{k+1 \choose 2} \prod_{i=k+1}^n (1+p^i).\end{eqnarray*}
Now set $p=q$ and apply Theorem~\ref{qbth}.
\end{proof}

\section*{Acknowledgments}
The research presented here began while both authors were visiting the Mittag-Leffler Institute as participants in a combinatorics program organized by Anders Bj\"orner and Richard Stanley.  We thank the Institute for its hospitality and support.  We are also grateful to Ira Gessel and Richard Stanley for some  useful discussions and 
references.


\begin{thebibliography}{xxxx}


\bibitem{bs}  E. Babson and E. Steingr\'{\i}msson, {\it Generalized permutation patterns
and a classification of the Mahonian statistics}, S\'em.  Lothar.
Combin., {\bf B44b} (2000), 18 pp.

\bibitem{br} D. Beck and J.B. Remmel, {\it Permutation enumeration of the symmetric group and the combinatorics of symmetric functions},  J. Combin. Theory Ser. A  {\bf 72}  (1995), 1--49.

\bibitem{bj} A. Bj\"orner, {\it Shellable and Cohen-Macaulay partially ordered sets},
Trans. AMS {\bf 260} (1980), 159--183.

\bibitem{bj2} A. Bj\"orner, {\it Posets, regular CW complexes and Bruhat order}, Europ. J. Combin. {\bf 5} (1984), 7--16.

\bibitem{blswz} A. Bj\"orner, M. Las Vergnas, B. Sturmfels, N. White, G. Ziegler, Oriented Matroids, Encyclopedia of Mathematics and its Applications, Cambridge University Press, 1993.

\bibitem{bw} A. Bj\"orner and V. Welker,
{\it Segre and Rees products of posets, with ring-theoretic applications},
J. Pure Appl. Algebra {\bf 198} (2005),  43--55.

\bibitem{Br} K. S. Brown, {\it Buildings}, Springer-Verlag, New
York, 1989.



\bibitem{c} L. Carlitz, {\it A combinatorial property of $q$-Eulerian numbers}, The American Mathematical Monthly, {\bf 82} (1975), 51--54.


\bibitem{csv} L. Carlitz, R. Scoville, and T. Vaughan, {\it Enumeration of pairs of sequences by rises, falls and levels}, Manuscripta Math. {\bf 19} (1976),  211--243.

\bibitem{csz} R.J. Clarke, E.  Steingr\'{\i}msson, and J. Zeng,  {\it New
Euler-Mahonian statistics on permutations and words}, Adv. in Appl.
Math.  {\bf 18}  (1997),   237--270.




\bibitem{foa3} D. Foata, {\it Distributions eul\'eriennes et mahoniennes
sur le groupe des permutations}, NATO Adv. Study Inst. Ser., Ser.
C: Math. Phys. Sci., 31, Higher combinatorics (Proc. NATO Advanced
Study Inst., Berlin, 1976), pp. 27--49, Reidel, Dordrecht-Boston,
Mass., 1977.





\bibitem{FH2} D. Foata and G.-N. Han, {\it Fix Mahonian calculus II;  further statistics}, Journal of Combinatorial Theory, Series A, Vol. 115, July 2008, p. 726-736.,

\bibitem{FH} D. Foata and G.-N. Han, {\it Fix Mahonian calculus III; A quadruple distribution}, Monatshefte fŸr Mathematik, 154, 2008, p. 177-197.



\bibitem{FH4} D. Foata and G.-N. Han, {\it Signed words and permutations V; a sextuple distribution}, the Ramanujan Journal, 2007.

\bibitem{FH3} D. Foata and G.-N. Han, {\it The q-tangent and q-secant numbers
via basic Eulerian polynomials},  preprint.




\bibitem{fs} D. Foata  and M.-P. Sch\"utzenberger, {\it Th\'eorie g\'eom\'etrique des polynomes eul\'eriens}, Lecture Notes in Mathematics, Vol. 138 Springer-Verlag, Berlin-New York 1970.

\bibitem{fs2} D. Foata  and M.-P. Sch\"utzenberger,
{\it Major index and inversion number of permutations},
Math. Nachr. {\bf 83} (1978), 143--159.

\bibitem{fz} D. Foata and D. Zeilberger, {\it Denert's permutation
statistic is indeed Euler-Mahonian},  Stud. Appl. Math.  {\bf 83}
(1990), 31--59.

\bibitem{fu} W. Fulton, {\it Introduction to toric varieties} Annals
of Mathematics Studies, 131, The William H. Roever Lectures in
Geometry, Princeton University Press, Princeton, NJ, 1993.

\bibitem{gar} A.M. Garsia, {\it On the maj and inv q-analogs of Eulerian polynomials}, J. Linear Multilinear Alg. {\bf 8} (1980), 21--34.


\bibitem{gg} A.M. Garsia and I. Gessel,
{\it Permutation statistics and partitions},
Adv. in Math. {\bf 31} (1979),  288--305.

\bibitem{grem} A.M. Garsia and J.B. Remmel, {Q-counting rook configuration and a formula of Frobenius}, J. Combin. Theory Ser. A {\bf 41} (1986), 246--275.


\bibitem{gr} I.M. Gessel and C. Reutenauer,
{\it Counting permutations with given cycle structure and descent set},
J. Combin. Theory Ser. A {\bf 64} (1993), 189--215.

\bibitem{hag} J. Haglund, {\it q-Rook polynomials and matrices over finite
fields}, Adv. in Appl. Math. {\bf 20} (1998), 450-487.


\bibitem{jo} J. Jonsson, {\it The Rees product of a Boolean algebra and a chain}, preprint.





\bibitem{mac1} P.A. MacMahon, Combinatory Analysis, 2 volumes,
Cambridge University Press, London, 1915-1916.  Reprinted by
Chelsea, New York, 1960.



\bibitem{mr} Muldoon and Readdy, {\it The Rees product of the cubical lattice with the chain}, Abstracts of the 1038th Meeting of the  AMS,  1038-05-215.


\bibitem{pr} C. Procesi,
{\it The toric variety associated to Weyl chambers},
Mots, 153--161,
Lang. Raison. Calc.,
Herms, Paris, 1990.

\bibitem{rrw} A. Ram, J. Remmel, and T. Whitehead, {\it Combinatorics of the $q$-basis of symmetric functions},  J. Combin. Theory Ser. A  {\bf 76}  (1996),   231--271.


\bibitem{ra} D. Rawlings,
{\it Enumeration of permutations by descents, idescents, imajor index, and basic components},
J. Combin. Theory Ser. A {\bf 36} (1984), 1--14.




\bibitem{Ro} M. Ronan, {\it Lectures on Buildings}, Academic Press, San Diego, 1989.

\bibitem{sw} J. Shareshian and M.L. Wachs, {\it  q-Eulerian polynomials: excedance number and major Index}, Electron. Res. Announc. Amer. Math. Soc. 13 (2007), 33--45.


\bibitem{ShWa} J. Shareshian and M.L. Wachs, {\it Eulerian quasisymmetric functions}, preprint, 2008.

\bibitem{ShWa2} J. Shareshian and M.L. Wachs, {\it Rees products and lexicographic shellability}, in preparation.




\bibitem{sk} M. Skandera,
{\it An Eulerian partner for inversions},
S\'em. Lothar. Combin. {\bf 46} (2001/02), Art. B46d, 19 pp. (electronic).


\bibitem{So} L. Solomon, The Steinberg character of a finite group with $BN$-pair, {\it Theory of Finite Groups (Symposium, Harvard Univ., Cambridge, Mass., 1968)}, 213-221, Benjamin, New York, 1969.


\bibitem{st} R.P. Stanley,
{\it Ordered structures and partitions}, Memoirs Amer. Math. Soc.  {\bf 119} (1972).

\bibitem{st1} R.P. Stanley,
{\it Binomial posets, M\"obius inversion, and permutation enumeration},
J. Combinatorial Theory Ser. A {\bf 20} (1976), 336--356.

\bibitem{st2} R.P. Stanley,  {\it Log-concave and unimodal sequences in
algebra, combinatorics, and geometry},  Graph theory and its
applications: East and West (Jinan, 1986),  500--535, Ann. New
York Acad. Sci., 576, New York Acad. Sci., New York, 1989.



\bibitem{st5} R.P. Stanley,  Enumerative combinatorics, Vol. 1,
2nd ed.,
Cambridge Studies in Advanced Mathematics, {\bf 49}, Cambridge University Press, Cambridge, 1997.



\bibitem{sun} S. Sundaram, {\it The homology representations of the symmetric group on Cohen-Macaulay subposets of the partition lattice}, Advances in Math. {\bf 104} (1994), 225--296.



\bibitem{wa2} M.L. Wachs,
{\it An involution for signed Eulerian numbers},
Discrete Math. {\bf 99} (1992), 59--62.



\bibitem{w1} M.L. Wachs, {\it Poset topology: tools and applications}, Geometric Combinatorics,  IAS/PCMI
lecture notes series (E. Miller, V. Reiner, B. Sturmfels,  eds.), {\bf 13} (2007), 497--615.


\end{thebibliography}
\end{document}